\documentclass[12pt,a4paper,reqno]{amsart}
\usepackage{amsfonts}
\usepackage{booktabs}
\usepackage{amsmath,amssymb,amsthm,amsxtra}
\usepackage{float}
\usepackage{amssymb}
\usepackage{mathabx}
\usepackage{bbm}
\usepackage[colorlinks, linkcolor=purple,anchorcolor=Periwinkle,
citecolor=Red,urlcolor=Emerald]{hyperref}
\usepackage[usenames,dvipsnames]{xcolor}
\usepackage{enumitem}
\usepackage{fullpage}
\usepackage{graphicx}
\usepackage{subfigure}
\usepackage{bookmark}
\usepackage{tikz}\usetikzlibrary{matrix}
\usepackage{url}
\usepackage{dsfont}
\usepackage[colorinlistoftodos]{todonotes}
\usepackage{tablists} \restorelistitem

\usetikzlibrary{decorations.markings}
\tikzset{->-/.style={decoration={  markings,  mark=at position #1 with
    {\arrow{>}}},postaction={decorate}}}
\tikzset{-<-/.style={decoration={  markings,  mark=at position #1 with
    {\arrow{<}}},postaction={decorate}}}

\usepackage{setspace}
\usepackage{extarrows}
\usepackage[all]{xy}
\usepackage{thmtools}
\usepackage{thm-restate}
\usepackage{hyperref}
\usepackage{cleveref}
\usepackage{multirow}
\usepackage{mathtools}
\usepackage{autobreak}
\usetikzlibrary{positioning}
\usepackage{xcolor}
\usepackage{graphicx}
\usetikzlibrary{decorations.pathreplacing}
\usetikzlibrary{calc}

\usepackage{amsthm}
\usepackage{cleveref}
\crefname{thm}{Theorem}{Theorems}
\crefname{asm}{Assumption}{Assumptions}
\crefname{cor}{Corollary}{Corollaries}
\crefname{dfn}{Definition}{Definitions}
\crefname{fct}{Fact}{Facts}
\crefname{lem}{Lemma}{Lemmas}
\crefname{mth}{Theorem}{Theorems}
\crefname{ntn}{Notation}{Notations}
\crefname{prp}{Proposition}{Propositions}
\crefname{rmk}{Remark}{Remarks}
\crefname{eg}{Example}{Examples}
\crefname{section}{\S\!}{\S\S\!}
\crefname{subsection}{\S\!}{\S\S\!}
\crefname{subsubsection}{\S\!}{\S\S\!}
\crefname{equation}{equation}{equations}

\theoremstyle{definition}
\newtheorem{thm}{Theorem}[section]
\newtheorem{dfn}[thm]{Definition}

\newtheorem{prp}[thm]{Proposition}

\newtheorem{rmk}[thm]{Remark}
\newtheorem{eg}[thm]{Example}

\newtheorem*{ntn*}{Notation}
\newtheorem*{rmk*}{Remark}

\newcommand{\ul}{\underline}

\newcommand{\xrr}[1]{\xrightarrow{\, #1 \, }}
\newcommand{\lto}{\longrightarrow}
\newcommand{\lmto}{\longmapsto}

\newcommand{\lsto}{\xrr{\sim}}

\newcommand{\bbA}{\mathbb{A}}
\newcommand{\bbC}{\mathbb{C}}

\newcommand{\bbH}{\mathbb{H}}

\newcommand{\calB}{\mathcal{B}}

\newcommand{\calE}{\mathcal{E}}
\newcommand{\calF}{\mathcal{F}}

\newcommand{\calI}{\mathcal{I}}
\newcommand{\calJ}{\mathcal{J}}
\newcommand{\calL}{\mathcal{L}}
\newcommand{\calM}{\mathcal{M}}

\newcommand{\calO}{\mathcal{O}}
\newcommand{\calP}{\mathcal{P}}

\DeclareMathOperator{\End}{End}
\DeclareMathOperator{\Ext}{Ext}
\DeclareMathOperator{\Hom}{Hom}

\DeclareMathOperator{\Ker}{Ker}

\DeclareMathOperator{\Hilb}{Hilb}

\title[Symplectic geometry of Higgs moduli and the Hilbert scheme of points]{Symplectic Geometry of Higgs Moduli and the Hilbert Scheme of Points over an Elliptic Curve}
\author{ZELIN JIA}
\date{\today}
\address{Graduate School of Mathematics, Nagoya University, Furocho, Chikusa-ku, Nagoya, 464-8602, Japan}
\email{zelin.jia.c0@math.nagoya-u.ac.jp}

\begin{document}
%%%%%%%%%%%%%%%%%%%%%%%%%%%%%%%%%%%%%%%%%%%%%%%%%%%%%%%%%%%%%%%%%%%%%%%%%%%%%%%%%%%%%%%%%%%%%%%%%%%%
\maketitle

\begin{abstract}
  We show that the isomorphism first obtained by \cite{GNR} between the moduli space of Higgs bundles $\calM_H^s(n,0)$ with a certain structure over the elliptic curve $C$ and the Hilbert scheme of $n$ points $\Hilb^n(T^*C)$ is a symplectomorphism with respect to their natural symplectic structures.

  This note consists mainly of two parts: 
  The first part is to show that, by using the Fourier-Mukai transform, we can construct the isomorphism between $\calM_H^s(n,0)$ and $\Hilb^n(T^*C)$, which is a review of \cite{GNR,MG}. 
  The second part is to show that such an isomorphism is a symplectomorphism with respect to their natural symplectic structures.
\end{abstract}

\tableofcontents

%%%%%%%%%%%%%%%%%%%%%%%%%%%%%%%%%%%%%%%%%%%%%%%%%%%%%%%%%%%%%%%%%%%%%%%%%%%%%%%%%%%%%%%%%%%%%%%%%%%%
%%%%%%%%%%%%%%%%%%%%%%%%%%%%%%%%%%%%%%%%%%%%%%%%%%%%%%%%%%%%%%%%%%%%%%%%%%%%%%%%%%%%%%%%%%%%%%%%%%%%
\section{Introduction \& Preliminaries}

The study of Higgs bundles has many interesting aspects in the theory of integrable systems.
Firstly, in Hitchin's original paper \cite{Hi}, he constructed a proper morphism from the moduli space of Higgs bundles to a vector space which has similar structures to the Liouville completely integrable systems \cite{Hi2}.

On the other hand, Nakajima illustrated in his book \cite{HK} that the Hilbert scheme of $n$ points over the cotangent bundle of a Riemann surface $\text{Hilb}^n(T^*\Sigma)$ is analogous in many ways to the moduli space of Higgs bundles.
So it is natural to ask whether $\text{Hilb}^n(T^*\Sigma)$ is related with an integrable system, and how is the integrable system of $\text{Hilb}^n(T^*\Sigma)$ related to the Hitchin system of $\calM_H(n,d)$.

In the case of an elliptic curve $C$, we indeed have several fruitful results around the relationship between the Hilbert scheme of $n$ points and the moduli space of Higgs bundles.

In \cite{GNR}, using gauge theory and Fourier transform, Gorsky, Nekrasov and Rubtsov showed that the Hilbert scheme $\text{Hilb}^n(T^*C)$ is isomorphic to the moduli space of rank $n$ degree $0$ semistable Higgs bundles with marked structure $\calM_H^v(n,0)$.
\[
  \calM_H^v(n,0)\cong \text{Hilb}^n(T^*C).
\]
Later, Groechenig \cite{MG} observed that the marked structure is the parabolic structure on Higgs bundles and extended their result to certain five families of isomorphism between $\Gamma$-Hilbert scheme of points and moduli spaces of parabolic Higgs bundles. 

Then it is natural to ask whether the behavior of symplectic structures on $\calM_H^v(n,0)$ and $\text{Hilb}^n(T^*C)$ stays the same or not?

Note that the symplectic structure on the moduli space of parabolic Higgs bundles $\calM_H^s(n,0)$ is constructed by Biswas and Ramanan \cite[Section 6]{BR} using deformation theory.
On the other hand, there exists a symplectic structure on $\text{Hilb}^n(T^*C)$ induced naturally by the symplectic structure of $T^*C$.

In \cite[Section 4.3]{Hu}, Hurtubise showed that over a compact Riemann surface $\Sigma$ of genus $g$, denote the moduli space of stable Higgs bundles of rank $n$ degree $d$ as $\calM_H(n,d)$, the dimension of $\calM_H(n,d)$ is $2n^2(g-1)+2$.
We have a symplectomorphism between $U \subset \calM_H(n,d)$ and $V \subset \text{Hilb}^n(T^*\Sigma)$ with respect to the above two symplectic structures.
In \cite[Section 2.2]{GNR}, the authors gave a review of the above symplectic equivalence using an explicit realization of the separation of variables.

However, over an elliptic curve $C$, the above result degenerates to a rather trivial case, that is, the symplectomorphism between:
\[
  T^*\text{Pic}^0(C)\cong T^*C.
\]

In this note, we show that the isomorphism $\calM_H^s(n,0)\cong \mathrm{Hilb}^n(T^*C)$ is indeed a symplectomorphism over the whole spaces with respect to their natural symplectic structures
(one is given by deformation theory and the other is induced by the symplectic structure of $T^*C$).

%%%%%%%%%%%%%%%%%%%%%%%%%%%%%%%%%%%%%%%%%%%%%%%%%%%%%%%%%%%%%%%%%%%%%%%%%%%%%%%%%%%%%%%%%%%%%%%%%%%%
\subsection{Notations \& Preliminaries}

Below is a list of some basic notation that we will need throughout this note.

\begin{itemize}[leftmargin=*]
  \item $C$ denotes a smooth projective algebraic curve of genus $1$ over the complex number field. 
  Choosing a point $o\in C$, we regard it as the origin of $C$.
  
  \item The Abel-Jacobi map gives an isomorphism between an elliptic curve and its dual variety,
  \[
    A_j\ :\ C\lto \text{Pic}^0(C)\ ;\ p\lmto \calO(p - o).
  \]
  We denote by $(\widehat{C}\coloneqq \text{Pic}^0(C),\widehat{o})$ the dual elliptic curve of $C$, where the distinguished point $\widehat{o}$ corresponds to the trivial line bundle $\calO$.
  We denote by $\widehat{p}$ the point of the dual elliptic curve corresponding to $\calO(p - o)$.

  \item By $\mathbf{R}$ we mean the right derived functor, and $\mathbf{R}^i$ its $i$-th cohomology. 
  Similarly, by $\mathbf{L}$ we mean the left derived functor, and $\mathbf{L}^i$ its $i$-th cohomology. 

  \item There exists a normalized Poincar\'{e} bundle $\calP$ on $C\times \widehat{C}$ such that we have $\calP|_{C\times \{\widehat{p}\}}=\calO_C(p-o)$ and $\calP|_{\{o\}\times \widehat{C}}$ is trivial. (See \cite[Section 2.5]{La} for details)

  \item Let $p_1$ and $p_2$ denote the projections of $C\times \widehat{C}$ to $C$ and $\widehat{C}$, a coherent sheaf $\calF$ on $C$ is called WIT-sheaf of index $i$ (WIT stands for weak index theorem), if 
  \[
    \mathbf{R}^jp_{2*}(\calP\otimes p_1^*\calF) = 0 \ \ \ \text{for all}\ j\neq i.
  \]
  Here the symbol $\otimes$ is an abbreviation of the symbol $\otimes_{\calO_{C\times \widehat{C}}}$. 
  (See \cite[Section 14.2]{La} for more details about the WIT-sheaf.)

  \item A Higgs bundle of rank $n$ and degree $d$ over $C$ is a pair $(E,\varphi)$, where $E$ is a vector bundle on $C$ and $\varphi$ is a map of vector bundles
  \[
    \varphi: E\lto E\otimes K_C
  \]
  called the Higgs field. 
  In our case, since the canonical bundle over an elliptic curve is trivial, the Higgs field is the endomorphism of $E$, i.e. $\varphi\in H^0(C,\End(E))$.

  \item Two Higgs bundles $(E,\varphi)$ and $(F,\phi)$ are said to be isomorphic if there exists an isomorphism $f: E\to F$ of vector bundles such that the following diagram commutes.
  \[
    \xymatrix{
      E \ar[r]^ - {\varphi} \ar[d]^ - {f} & E\otimes K_C \ar[d]^ -{f\otimes \text{id}} \\
      F \ar[r]^ -{\phi} & F\otimes K_C
    }
  \]
\end{itemize}

Here we introduce some basic definitions and facts around Higgs bundles.

\begin{dfn}
  Given a vector bundle $E$ on $C$, the slope of $E$ is defined by 
  \[
    \mu(E) =\deg (E) \slash \text{rank}(E).
  \]
  where $\text{rank}(E)$ is the rank of $E$, and $\deg (E)$ is the degree of $E$.
\end{dfn}

\begin{dfn}
  Given a Higgs bundle $(E,\varphi)$, we say that a subbundle $F\subset E$ is $\varphi$-invariant if $\varphi(F)\subseteq F\otimes K_C$. 
  We can view $(F,\varphi|_F)$ as a Higgs subbundle of $(E,\varphi)$.
  A Higgs bundle $(E,\varphi)$ is semistable if the slope of any $\varphi$-invariant subbundle $F$ satisfies 
  \[
    \mu(F)\leq \mu(E).
  \]
  A Higgs bundle is stable if the above inequality is strict for every proper Higgs subbundle.

  A Higgs bundle is polystable if it is semistable and isomorphic to a direct sum of stable Higgs bundles of the same slope.
\end{dfn}

If $(E,\varphi)$ is a semistable Higgs bundle of slope $\mu$, then a general fact tells us that it has a Jordan-H\"{o}lder filtration of Higgs subbundles (See \cite[Proposition 4.1]{Ni})
\[
  0 = E_0\subsetneq E_1 \subsetneq E_2 \subsetneq \cdots \subsetneq E_r = E
\]
where the Higgs subbundles $(E_i\slash E_{i-1},\varphi|_{E_i\slash E_{i-1}})$ have the same slope $\mu$.

Using Jordan-H\"{o}lder filtration, we can define the graded object associated to every semistable Higgs bundle $(E,\varphi)$, that is 
\[
  \text{Gr}(E,\varphi)\coloneqq \oplus_i (E_i\slash E_{i-1},\varphi|_{E_i\slash E_{i-1}})
\]
one can show that the graded object $\text{Gr}(E,\varphi)$ associated to $(E,\varphi)$ is well defined up to isomorphism.
We say that two semistable Higgs bundles $(E,\varphi)$ and $(F,\phi)$ are $S$-equivalent if $\text{Gr}(E,\varphi)\cong \text{Gr}(F,\phi)$.

Indeed, there is a unique polystable Higgs bundle representing the $S$-equivalent class of every semistable Higgs bundle.

\begin{thm}\cite{Hi,Ni,Sim}
  The moduli space $\calM_H^{ss}(n,d)$ whose closed points parametrize semistable Higgs bundles is a quasi-projective algebraic variety.

  The moduli space $\calM_H(n,d)$ whose closed points parametrize stable Higgs bundles is open and dense in $\calM_H^{ss}(n,d)$, moreover, $\calM_H(n,d)$ is a smooth projective algebraic variety.
\end{thm}

In our case over an elliptic curve, we have the following theorem.

\begin{thm}\cite[Theorem 4.19]{FGN}
  A Higgs bundle $(E,\varphi)$ over $C$ is semistable if and only if $E$ is semistable.
  In particular, there are no stable Higgs bundles over $C$ of rank $n>1$, degree $d=0$, and for the moduli space of semistable Higgs bundles over $C$, we have 
  \[
    \calM_H^{ss}(n,d)\cong \text{Sym}^n(T^*C)
  \]
  where the degree $d=0$.
\end{thm}

In the main text, we will study the Higgs bundles equipped with extra structure.

\begin{dfn}
  Let $\calM_H^v(n,d)$ denote the moduli space of degree $d$ rank $n$ marked Higgs bundles on $C$, whose closed points parametrize triples $(E,\varphi,v)$ where $(E,\varphi)$ is a semistable Higgs bundle and $v\in E_o$ such that there are no proper $\varphi$-invariant subbundle $F\subset E$ with $\mu(F)\geq \mu(E)$ and $v\in F_o$.
\end{dfn}

\begin{dfn}
  Let $\calM_H^s(n,d)$ denote the moduli space of degree $d$ rank $n$ stable parabolic Higgs bundles $(E_*,\varphi)$ over the elliptic curve $C$,
  where $E_*\coloneqq (E,L_o)$ with $L_o$ a line of the fiber $E_o$ over the origin $o\in C$.
  The Higgs field $\varphi$ can have at worst a first order pole at $o$, and the residue of the Higgs field at $o$ is nilpotent with respect to its multi-dimensional $\{n,1,0\}$ flag.

  In other words, the Higgs field $\varphi$ is the bundle morphism of the form $\varphi: E\lto E\otimes K_C\otimes \calO_C(o)$, and the image of the homomorphism 
  \[
    \varphi_o : E_o \lto (E\otimes \calO_C(o))_o
  \]
  is contained in the subspace $L_o\otimes (\calO_C(o))_o$ and furthermore, we have $\varphi_o(L_o)=0$.

  The stability is given as for every $\varphi$-invariant subbundle $F\subset E$, we have $\mu(F)< \mu(E)$. 
\end{dfn}

Indeed, \cite[Section 2.5.]{MG} obtained a natural equivalence between the parabolic structures and the marked structures. %which we will give a review in this section. 
Consequently, one can regard $\calM_H^s(n,d)$ and $\calM_H^v(n,d)$ as the same moduli spaces.

%%%%%%%%%%%%%%%%%%%%%%%%%%%%%%%%%%%%%%%%%%%%%%%%%%%%%%%%%%%%%%%%%%%%%%%%%%%%%%%%%%%%%%%%%%%%%%%%%%%%
%%%%%%%%%%%%%%%%%%%%%%%%%%%%%%%%%%%%%%%%%%%%%%%%%%%%%%%%%%%%%%%%%%%%%%%%%%%%%%%%%%%%%%%%%%%%%%%%%%%%
\section{Marked Higgs bundles and Hilbert scheme of $n$ points}

In this section, we give a review of the isomorphism $\calM_H^s(n,0)\cong \text{Hilb}^n(T^*C)$.

%===================================================================================================
\subsection{The BNR correspondence}

Consider any rank $n$ Higgs bundle $(E,\varphi)$ over $C$, we can associate to it a polynomial called the characteristic polynomial.
\[
  P(\lambda) =\sum_{i = 0}^{n}s_i\lambda^{n-i},\ \ (s_i=(-1)^i\text{tr}(\wedge^i\varphi),\lambda\in K_C)
\]

Denote the total space of $K_C$ as $|K_C|$, and $p\ :\ |K_C|\lto C$ the canonical projection.
Consider the pullback of $K_C$ onto its total space, together with its tautological section $\lambda \in H^0(|K_C|,p^*K_C)$, i.e. $\lambda(v_x)=v_x\in (p^*K_C)_x,\ \forall x\in C$.

Then the pullback bundle $p^*K_C$ is such that its fiber over $v\in |K_C|$ is the fiber over $p(v)$, i.e., $(p^*K_C)_v=(K_C)_{p(v)}$.

\begin{dfn}
 Denote the Hitchin base $\calB = \bigoplus_{i = 1}^{n}H^0(C,(K_C)^i)$, set $ s = (s_1, s_2, \dots s_n) \in \calB$.
 Then the spectral curve $X_s\subset T^*C$ is given by the zero locus in $|K_C|$ of the section $P_s(\lambda)$ :
  \[
    P_s(\lambda) \coloneqq \lambda^n + (p^*s_1)\lambda^{n - 1} +(p^*s_2)\lambda^{n - 2} +\cdots + p^*s_n \in H^0(|K_C|,p^*(K_C)^n).
  \]
  We denote the natural projection map from the spectral curve $X_s$ to $C$ as $\pi\ :\ X_s\lto C$.
\end{dfn}

\begin{prp}[\cite{BNR}]
  Define the morphism $u$ of sheaves as
  \[
    u\ :\ K_C^{- n}\lto \calO_C\oplus K_C^{ - 1}\oplus \cdots \oplus K_C^{ - n} ;\ \alpha \lmto \alpha(s_n + s_{n - 1} +\cdots + 1).
  \]
  If we denote $\calJ_s$ as the ideal sheaf of $\text{Sym}(K_C^{-1})$ generated by the image $\text{Im}(u)$, then the scheme structure of the spectral curve $X_s$ is given as $X_s= \text{\ul{Spec}}(\text{Sym}(K_C^{-1})/\calJ_s)$.
\end{prp}

From the above proposition, one can conclude that the natural projection $\pi\ :\ X_s\lto C$ is finite, hence affine.
Since an affine morphism induces an equivalence of groupoids of quasi-coherent modules, we have the following theorem.

\begin{thm}[\cite{BNR}]\label{BNR}
  Let $s\in \calB$, and $X_s$ is the corresponding spectral curve, and we denote the characteristic polynomial which defines the spectral curve $X_s$ as $P_s$.
  Then there is a natural equivalence of groupoids between
  \begin{itemize}
    \item Quasi-coherent sheaf $\calL$ on $X_s$ with $\chi(\calL)=d+n(g-1)$ and $\pi_{*}\calL$ is locally free of rank $n$.
    \item Higgs bundle $(E,\varphi)$ on $C$ of rank $n$ degree $d$ such that the characteristic polynomial of $\varphi$ equals $P_s$.
  \end{itemize}
\end{thm}

%=============================================================================================
\subsection{Higgs moduli as Hilbert scheme of $n$ points}

Now we use the BNR correspondence and the relative Fourier-Mukai transform to prove the following theorem.

\begin{thm}\cite[Theorem 4.1]{MG}\label{thm1}
  Recall that the moduli space $\calM_H^s(n,0)$ parametrizes triples $(E,\varphi,v)$ where $(E,\varphi)$ is a degree $0$ rank $n$ semistable Higgs bundle and $v\in E_o$ such that there are no proper $\varphi$-invariant subbundle $F\subset E$ with $\mu(F)\geq \mu(E)$ and $v\in F_o$.

  There is an isomorphism between $\calM_H^s(n,0)$ and the Hilbert scheme of $n$ points over $T^*C$.
  \[
  \calM_H^s(n,0)\cong \text{Hilb}^n(T^*C).
  \]
\end{thm}

\begin{proof}[Sketch of proof]
  We have the Fourier-Mukai transform with kernel $\calP$
  \begin{align*}
    \Phi\ :\ D^b(C) &\lto D^b(\widehat{C}) \\
    \calE^{\bullet} &\lmto \mathbf{R}p_{2_*}(\mathbf{L}p_1^*\calE^{\bullet}\otimes^\mathbf{L}\calP) = \mathbf{R}p_{2_*}(p_1^*\calE^{\bullet}\otimes\calP)
  \end{align*}
  
  Since the cotangent bundle on an elliptic curve is trivial, we can obtain the projection from the cotangent bundle $f:T^*C\lto C$ via the base change:
  \begin{align*}
    \xymatrix{
      T^*C\cong C\times \bbA^1 \ar[d] \ar[r]& C \ar[d] \\
      \bbA^1 \ar[r]& \text{Spec}\bbC
    }
  \end{align*}
  Similarly we can regard the projection from the cotangent bundle $g:T^*\widehat{C}\lto \widehat{C}$ as the base change.
  From the universality of the pullback diagram, we obtain the following commutative diagram:
  \begin{align*}
    \xymatrix{
      &T^*C\times_{\bbA^1} T^*\widehat{C} \ar[rr] \ar[ld] \ar[dd] & & T^*C \ar[ld] \ar[dd] & \\
       \widehat{C}\times C \ar[dd] \ar[rr] & & C \ar[dd] \\
      &T^*\widehat{C} \ar[rr] \ar[ld] & & \bbA^1 \ar[ld] & \\
      \widehat{C} \ar[rr] & & \text{Spec}\bbC
    }
  \end{align*}
  Denote the newly appearing morphisms in the above diagram as $l: T^*C\times_{\bbA^1} T^*\widehat{C}\lto C\times \widehat{C}$, $\pi_1: T^*C\times_{\bbA^1} T^*\widehat{C}\lto T^*C$ and $\pi_2: T^*C\times_{\bbA^1} T^*\widehat{C}\lto T^*\widehat{C}$.
  
  The relative Fourier-Mukai transform obtained from the above base change diagram 
  \begin{align}\label{relative Fourier-Mukai}
    \Phi_{\bbA^1}\ :\ D^b(T^*C)&\lto D^b(T^*\widehat{C}) \\
    \calE^{\bullet} &\lmto \mathbf{R}\pi_{2_*}(\mathbf{L}\pi_1^*\calE^{\bullet}\otimes^\mathbf{L}l^*\calP) = \mathbf{R}\pi_{2_*}(\pi_1^*\calE^{\bullet}\otimes l^*\calP) \notag
  \end{align}
  is an equivalence of categories, with kernel $l^*\calP\in D^b(T^*C\times_{\bbA^1}T^*\widehat{C})$.

  Consider the structure sheaf $\calO_Z$ of a length $n$ subscheme $Z\subset T^*C$, $\calF \coloneqq \pi_1^*\calO_Z\otimes l^*\calP$, $\calM \coloneqq \Phi_{\bbA^1}(\calO_Z)=\mathbf{R}\pi_{2_*}(\calF)$ is a sheaf complex.
  However, $\pi_{2}: \text{Supp}\calF \lto T^*\widehat{C}$ is finite, so the higher direct image $\mathbf{R}^i\pi_{2_*}(\calF)\ \ (i\neq 0)$ vanishes, thus $\calM$ is actually a sheaf. 
  
  Moreover, we can show that $g_*\calM$ is locally free of rank $n$ and $\text{Supp}\calM$ corresponds to a spectral curve $X_s\subset T^*\widehat{C}$.
  Thus, from the BNR correspondence (\cref{BNR}), via the direct image $g_*$, $\calM$ corresponds to a rank $n$ degree $0$ Higgs bundle $(M(\coloneqq g_*\calM),\varphi)$.

  To determine the marked structure of our Higgs bundle $(M,\varphi)$, consider the following exact sequence.
  \[
    0\lto \calI_Z \lto \calO_{T^*C} \xrr{\ s\ } \calO_Z \lto 0.
  \]
  Applying the relative Fourier-Mukai transform, we get the exact triangle in $D^b(T^*\widehat{C})$,
  \[
    \calI \lto \calO_{g^{ - 1}({\widehat{o}})}[ - 1] \lto \calM \lto \calI[1].
  \]
  The datum of the surjection $s: \calO_{T^*C}\lto \calO_Z$ is translated by the relative Fourier-Mukai functor into an element in 
  \begin{align*}
    \Hom_{D^b(T^*\widehat{C})}(\calO_{g^{ - 1}(\widehat{o})}[ - 1],\calM) &\cong \Hom_{D^b(T^*\widehat{C})}(g^*\bbC_{\widehat{o}},\calM[1]) \cong \Hom_{D^b(\widehat{C})}(\bbC_{\widehat{o}},M[1]) \\
    &\cong \Ext^1_{\widehat{C}}(\bbC_{\widehat{o}},M) \cong \Hom(M,\bbC_{\widehat{o}})^* \cong M_{\widehat{o}}.
  \end{align*}
  Thus, the surjection $s$ is parametrized by the vector $v\in M_{\widehat{o}}$, where $M_{\widehat{o}}$ is denoted as the fiber of $M$ over the origin $\widehat{o}\in \widehat{C}$.

  Finally, the stability condition for the marked Higgs bundle is that the Higgs bundle $(M,\varphi)$ has no proper Higgs subbundle $(N,\varphi_{|N})$ of degree $0$ such that $v\in N_{\widehat{o}}$.
\end{proof}

\begin{eg}\label{ex1}
  The moduli space of rank $1$ degree $0$ Higgs bundles over the elliptic curve $C$ which we denote by $\calM_H(1,0)$ is isomorphic to the cotangent bundle: 
  \[
    T^*\widehat{C} \cong \calM_H(1,0)
  \]
  Indeed, $\widehat{C}$ parametrizes degree $0$ line bundles, and for any degree $0$ line bundle $L$ which corresponds to a point $\widehat{p}\in \widehat{C}$, we have
  \[
    H^0(C,\End L\otimes K_C)\cong H^0(C,K_C) \cong H^1(C,\calO_C)^* \cong T^*_{L}\text{Pic}^0(C) = T^*_{\widehat{p}}\widehat{C}
  \]
  so we have the isomorphism given by 
  \begin{align*}
    \calM_H(1,0) &\lsto T^*\widehat{C}( =\widehat{C}\times \bbA^1) \\
    (L,\varphi_t(\coloneqq t \text{d}z)) &\lmto (\widehat{p},t)
  \end{align*}

  Now, fix the degree $0$ line bundle $L$ which corresponds to a point $\widehat{p}\in \widehat{C}$, we denote the ``horizontal'' spectral curve given by the section $t\text{d}z\in H^0(C,K_C)$ as $X_t\subset T^*C$, and denote the natural projection as $\pi_t: X_t\lto C$.

  Via the BNR correspondence, the rank $1$ degree $0$ Higgs bundle $(L,\varphi_t)$ corresponds to the line bundle $L_t$ on $X_t$, i.e. $\pi_{t*}L_t=L$.
  On the other hand, if we denote $\widehat{\pi}_t: \widehat{X_t}\lto \widehat{C}$ similarly, then we have $\widehat{\pi}_{t*}(\bbC_{(\widehat{p},t)})=\bbC_{\widehat{p}}$.

  In this case, the relative Fourier-Mukai transform $\Phi_{\bbA^1}$ degenerates to $\Phi_t: D^b(\widehat{X_t})\lto D^b(X_t)$, by the base change formula of Fourier-Mukai transform, we have 
  \[
    \pi_{t*}\Phi_t(\bbC_{(\widehat{p},t)}) \cong \Phi(\widehat{\pi}_{t*}(\bbC_{(\widehat{p},t)})) =\Phi(\bbC_{\widehat{p}})\cong L.
  \]
  Therefore, we have seen that the Higgs bundle $(L,\varphi_t)$ corresponds to the skyscraper sheaf $\bbC_{(\widehat{p},t)}\subset T^*\widehat{C}$.
\end{eg}

Since the elliptic curve $C$ is isomorphic to its dual $\widehat{C}$, we will make no distinction between them later on.

%%%%%%%%%%%%%%%%%%%%%%%%%%%%%%%%%%%%%%%%%%%%%%%%%%%%%%%%%%%%%%%%%%%%%%%%%%%%%%%%%%%%%%%%%%%%%%%%%%%%
%%%%%%%%%%%%%%%%%%%%%%%%%%%%%%%%%%%%%%%%%%%%%%%%%%%%%%%%%%%%%%%%%%%%%%%%%%%%%%%%%%%%%%%%%%%%%%%%%%%%
\section{Symplectic structures on marked Higgs bundles and Hilbert scheme of $n$ points}

In this section, we show that the isomorphism in \cref{thm1} is indeed a symplectomorphism, i.e.
the pullback of the natural symplectic structure on $\text{Hilb}^n(T^*C)$ coincides with the natural symplectic structure on $\calM_H^v(n,0)$.

%%%%%%%%%%%%%%%%%%%%%%%%%%%%%%%%%%%%%%%%%%%%%%%%%%%%%%%%%%%%%%%%%%%%%%%%%%%%%%%%%%%%%%%%%%%%%%%%%%%%
\subsection{Symplectic structure on the moduli space of marked Higgs bundles}\label{Symp}

For any semistable Higgs bundle $(E,\varphi)$, consider the End-bundle $\End(E)$ and its associated Dolbeault operator $\bar{\partial}$.
We have the following diagram of the Dolbeault resolution of $\End(E)$

\[
  \xymatrix{
    & 0 \ar[d] & 0 \ar[d] \\
    0\ar[r]& \Omega^0(C,\End(E)) \ar[d]^-{\bar{\partial}} \ar[r]^ -{[\varphi, -]} & \Omega^{1,0}(\End(E)) \ar[d]^-{ -\bar{\partial}} \ar[r]& 0 \\
    0 \ar[r] & \Omega^{0,1}(\End (E)) \ar[d] \ar[r]^ -{[\varphi, -]} & \Omega^{1,1}(\End(E)) \ar[d] \ar[r] & 0 \\
    & 0 & 0
  }
\]
This gives us the complex 
\[
  0 \lto \Omega^0(C,\End(E))\xrr{\ f\ } \Omega^{0,1}(\End(E))\oplus \Omega^{1,0}(\End(E))\xrr{\ g\ } \Omega^{1,1}(\End(E))\lto 0
\]
where $f\coloneqq [\varphi, - ] +\bar{\partial}$, and $g\coloneqq [\varphi, - ] -\bar{\partial}$, denote the hypercohomology group $\Ker g\slash \text{Im}f$ as $\bbH^1(\End(E))$.
In fact, we have a homomorphism $\iota: \bbH^1(\End(E))\lto H^1(C,\End(E))$.

From \cite[Section 6]{BR}, the moduli space $\calM_H^v(n,d)$ has a natural holomorphic symplectic structure which we give a brief review here :

The tangent space at $(E,\varphi,v)$ is given as 
\begin{align}\label{tangent space}
  T_{(E,\varphi,v)} \calM_H^v(n,d) \cong \bbH^1(\End_o(E)),
\end{align}
where $\End_o(E)$ is given as $\End_o(E) = \{f\in \End(E) \mid f_{o}(v)=c \times v, c\in \bbC \}$.

Using Serre duality, the Liouville 1-form $\Omega_{(n,d)}$ is given by 
\begin{align}\label{1-form}
  \Omega_{(n,d)}(\beta) = \int_{C}\text{tr}(\iota(\beta)\varphi)
\end{align}
where $\beta\in T_{(E,\varphi,v)} \calM_H^v(n,d) \cong \bbH^1(\End_o(E))$. 

The differential of the Liouville 1-form can be constructed and is a symplectic form on $\calM_H^v(n,d)$.
We will denote such symplectic form by $\text{d}\Omega_{(n,d)}$.

%========================================================================================
\subsection{Symplectic structure on the Hilbert scheme of $n$ points}\label{sec4}

The generic length $n$ subscheme of the Hilbert scheme $\text{Hilb}^n(T^*C)$ is given by $n$ pairwise distinct points $(p_1,t_1),\dots ,(p_n,t_n)\in T^*C$, we denote the set of such points as $(T^*C)^n_{\star}$, it is easy to see that this is an open subset of $\text{Hilb}^n(T^*C)$, and its inverse image by the quotient map $(T^*C)^n\lto \text{Sym}^n(T^*C)$ is denoted by $(T^*C)_{\circ}^n$.

The next theorem shows that there exists a holomorphic symplectic form $\text{d}\Omega$ on $\text{Hilb}^n(T^*C)$ induced uniquely by the natural symplectic form on $(T^*C)^n_{\star}$.

\begin{thm}(\cite{Bea}, see also \cite[Theorem 1.17]{HK})\label{thm3}
  $\text{Hilb}^n(T^*C)$ has a holomorphic symplectic form which is induced by the holomorphic symplectic form on $(T^*C)^n_{\star}$ uniquely.
\end{thm}

\begin{proof}
  Let $\text{Sym}^n_1(T^*C)$ be the subset of $\text{Sym}^n(T^*C)$ consisting of $\sum \nu_i[q_i]$ ($q_i$ distinct) with $\nu_1\leq 2,\ \nu_2=\cdots=\nu_k=1$.
  Its inverse image by the Hilbert-Chow morphism $\pi: \text{Hilb}^n(T^*C)\lto \text{Sym}^n(T^*C)$ is denoted by $(T^*C)_1^{[n]}$, and its inverse image by the quotient map $(T^*C)^n\lto \text{Sym}^n(T^*C)$ is denoted by $(T^*C)_1^n$. 
  Let us denote by $\triangle \subset (T^*C)^n$ the ``big diagonal'' consisting of elements $(q_1,\dots,q_n)$ with $q_i = q_j$ for some $i\neq j$.
  Then $\triangle \cap (T^*C)_1^n$ is smooth of codimension $2$ in $(T^*C)_1^n$. Moreover, we have the following commutative diagram
  \[
    \xymatrix{
      \text{Blow}_{\triangle} ((T^*C)_1^n) \ar[d]^ - {\rho} \ar[r]^ -{\eta} & (T^*C)_1^n \ar[d] \\
      (T^*C)_1^{[n]} \ar[r]^ - {\pi} & \text{Sym}_1^n(T^*C)
    }
  \]
  where $\rho$ is the map given by taking the quotient by the action of the symmetry group $S_n$. 
  It follows immediately that $\rho$ is a covering ramified along the exceptional divisor $D$ of the blow up $\eta$.

  Since $(T^*C)_1^n \setminus (T^*C)_{\circ}^n$ is of codimension $2$ in $(T^*C)_1^n$, by Hartogs theorem, the holomorphic symplectic form on $(T^*C)_{\circ}^n$ extends uniquely on $(T^*C)_1^n$, which we denote by $\text{d}\Omega$.
  
  The pull-back $\eta^*\text{d}\Omega$ is invariant under the action of $S_n$, hence it defines an unique holomorphic two form $\Lambda$ on $(T^*C)_1^{[n]}$ with $\rho^*\Lambda= \eta^*\text{d}\Omega$.
  Then we have 
  \[
    \text{div}(\rho^*(\Lambda^n)) = \rho^*\text{div}(\Lambda^n) + D.
  \]
  On the other hand, the left hand side is equal to 
  \[
    \text{div}(\eta^*(\text{d}\Omega^n)) = \eta^*\text{div}(\text{d}\Omega^n) + D = D.
  \]
  Therefore we have $\text{div}(\Lambda^n)=0$, hence $\Lambda$ is a holomorphic symplectic form on $(T^*C)_1^{[n]}$.

  Now, $\text{Hilb}^n(T^*C)\setminus (T^*C)_1^{[n]}$ is of codimension $2$ in $\text{Hilb}^n(T^*C)$, hence $\Lambda$ extends to the whole $\text{Hilb}^n(T^*C)$ uniquely as a holomorphic form by the Hartogs theorem.
  We still have $\text{div}\Lambda^n=0$ in $\text{Hilb}^n(T^*C)$, hence $\Lambda$ is non-degenerate.
\end{proof}

Let $\Omega$ denote the corresponding holomorphic one-form on $\text{Hilb}^n(T^*C)$, i.e. its differential is $\text{d}\Omega$.
From \cite[Proposition 3.2]{BM2}, such form $\Omega$ exists and is induced naturally by the tautological one form on $(T^*C)^n_{\star}$.

%%%%%%%%%%%%%%%%%%%%%%%%%%%%%%%%%%%%%%%%%%%%%%%%%%%%%%%%%%%%%%%%%%%%%%%%%%%%%%%%%%%%%%%%%%%%%%%%%%%%
\subsection{Symplectic structures in the generic case}\label{sec5}

This section is devoted to show that $\calM_H^v(n,0)\cong \text{Hilb}^n(T^*C)$ is a symplectomorphism with respect to their natural symplectic structures.

In the generic case, the underlying Higgs bundle of an element in $\calM_H^v(n,0)$ is polystable, thus the Higgs bundle is given by the direct sum $(E,\varphi)=\bigoplus_i(L_i,\varphi_i)$ with $n$ pairwise distinct pieces: 
\[
(L_i,\varphi_i)\neq (L_j,\varphi_j),\ (L_i,\varphi_i)\in \calM_H(1,0)\ (1\leq i\leq n).
\]

\begin{dfn}
  Define the moduli space $\calM_H^{ps}(n,0)$ as the subset of $\calM_H^v(n,0)$ which parametrizes the triples $(E,\varphi,v)$, where the underlying Higgs bundle is given by the direct sum $(E,\varphi)=\bigoplus_i(L_i,\varphi_i)$ with $n$ pairwise distinct pieces: 
  \[
    (L_i,\varphi_i)\neq (L_j,\varphi_j),\ (L_i,\varphi_i)\in \calM_H(1,0)\ (1\leq i\leq n).
  \]
  
  The marked structure is constructed by choosing a vector $v$ in the fiber $E_o$ which is not contained in any subspace given by direct sums of the sections of line bundles $L_i$.
\end{dfn}

The isomorphism class of marked Higgs bundles in this case is independent of the choice of the vector $v \in E_o$.
Therefore, two elements $(E,\varphi,v)$ and $(F,\psi,w)$ in $\calM_H^{ps}(n,0)$ are equivalent if and only if the underlying Higgs bundles are isomorphic $(E,\varphi) \cong (F,\psi)$, so $\calM_H^{ps}(n,0)$ indeed parametrizes the isomorphism classes of such Higgs bundles.

Moreover, since the Higgs bundles in $\calM_H^{ps}(n,0)$ are given by the direct sum of $n$ pairwise distinct components, to say that such two Higgs bundles are isomorphic is equivalent to say that their components are the same except the order.

\begin{prp}
  The isomorphism $\calM_H^v(n,0)\cong \text{Hilb}^n(T^*C)$ restricts to an isomorphism $\calM_H^{ps}(n,0) \cong (T^*C)^n_{\star}$, where $(T^*C)^n_{\star}$ is the reduced locus of $\text{Hilb}^n(T^*C)$.
\end{prp}

\begin{proof}
  It follows from \cref{ex1} and the fact that the relative Fourier-Mukai functor (\ref{relative Fourier-Mukai}) is an additive functor.
\end{proof}

Using notation similar to that in \cref{ex1}, we denote the isomorphism as 
\begin{align*}
  \calM_H^{ps}(n,0) &\longrightarrow  (T^*C)^n_{\star} \\
  \bigoplus_{i=1}^{n}(L_i,\varphi_i) &\longmapsto (p_1,t_1),\dots ,(p_n,t_n).
\end{align*}

\begin{prp}
  The isomorphism $\calM_H^{ps}(n,0) \cong (T^*C)^n_{\star}$ is a symplectomorphism.
\end{prp}

\begin{proof}
  Recall that $\calM_H^{ps}(n,0)$ parametrizes the isomorphism class of $(E,\varphi)=\bigoplus_i(L_i,\varphi_i)$ with $n$ pairwise distinct pieces.
  In this case, the choice of the marked structure does not affect the isomorphism class, therefore, the deformation at $(E, \varphi,v)$ reduces to
  \begin{align*}
    T_{(E,\varphi)} \calM_H^{ps}(n,0) \cong \bbH^1(\End(E)),
  \end{align*}
  for any $(E,\varphi) \in \calM_H^{ps}(n,0)$.

  We calculate the endomorphism $\End(E)$ in our case as
  \begin{align*}
    \End(E) \cong (\bigoplus_{i=1}^{n} L_i) \otimes (\bigoplus_{i=1}^{n} L_i)\ \check{} \cong \bigoplus_{i,j} (L_i \otimes L_j \check{}\ ).
  \end{align*}
  When $i \neq j$, $L_i \otimes L_j \check{}\ $ is a nontrivial degree $0$ line bundle on the elliptic curve $C$, thus we have $H^0(C, L_i \otimes L_j \check{}\ ) = H^1(C, L_i \otimes L_j \check{}\ )=0$.
  On the other hand, when $i = j$, we have $L_i \otimes L_j \check{}\ = \calO_C$. 

  Hence, the tangent space over the open locus $\calM_H^{ps}(n,0) \subset \calM_H^v(n,0)$ is given as 
  \begin{align*}
    T_{(E,\varphi)} \calM_H^{ps}(n,0) \cong \bbH^1(\End(E)) \cong \bbH^1(\bigoplus_{i,j} (L_i \otimes L_j \check{}\ )) \cong \bigoplus_{i=1}^{n} (H^1(C, \calO_C)_i \oplus H^0(C, \calO_C))_i .
  \end{align*}

  Now, choose the identification $H^1(C, \calO_C)_i \cong T_{p_i}C$, $H^0(C, \calO_C)_i \cong T_{t_i} \bbA^1$.
  We can read off the Liouville $1$-form in (\ref{1-form}) as
  \begin{align}
    \Omega_{(n,0)}(\beta) = \int_{C}\text{tr}(\iota(\beta)\varphi) = \int_{C}\sum_{i=1}^{n}(\iota(\beta)\varphi_i) = \sum_{i=1}^{n}(t_i\,dx_i)
  \end{align}
  where $\beta\in T_{(E,\varphi)} \calM_H^{ps}(n,0) \cong \bbH^1(\End(E))$. 

  This corresponds exactly to the Liouville $1$-form on the reduced locus $(T^*C)^n_{\star}$.
\end{proof}

From the proof of \cref{thm3}, the symplectic structure on $\text{Hilb}^n(T^*C)$ is obtained by extending the symplectic structure on $(T^*C)^n_{\star}$ uniquely. 
Hence, we can see that such symplectomorphism can be extended to the entire space.

\begin{thm}\label{thm5}
  $\calM_H^v(n,0)\cong \text{Hilb}^n(T^*C)$ is a symplectomorphism induced by the symplectomorphism $\calM_H^{ps}(n,0) \cong (T^*C)^n_{\star}$.
\end{thm}

\begin{rmk}
  It is well-known that the Hilbert Chow morphism $H: \text{Hilb}^n(T^*C)\lto \text{Sym}^n(T^*C)$ is the symplectic resolution.
  Together with \cref{thm5}, one can conclude that the ``Hilbert Chow morphism'' on the Higgs side is given as 
  \begin{align*}
    H_v: \calM_H^v(n,0) &\lto \calM_H^{ss}(n,0)\\
    (E,\varphi,v) &\lmto (E,\varphi)
  \end{align*}
  This is the map that forgets the datum of the vector $v\in E_o$.
  
  Finally, we have the following commutative diagram that preserves the symplectic structures.
  \[
    \xymatrix{
      \calM_H^v(n,0) \ar[r]^ - {\cong} \ar[d]^ - {H_v} & \text{Hilb}^n(T^*C)\ar[d]^ -{H} \\
      \calM_H^{ss}(n,0) \ar[r]^ - {\cong} & \text{Sym}^n(T^*C)
    }
  \]
\end{rmk}

%%%%%%%%%%%%%%%%%%%%%%%%%%%%%%%%%%%%%%%%%%%%%%%%%%%%%%%%%%%%%%%%%%%%%%%%%%%%%%%%%%%%%%%%%%%%%%%%%%%%
%%%%%%%%%%%%%%%%%%%%%%%%%%%%%%%%%%%%%%%%%%%%%%%%%%%%%%%%%%%%%%%%%%%%%%%%%%%%%%%%%%%%%%%%%%%%%%%%%%%%

\end{document}